\documentclass[12pt, reqno]{amsart}
\usepackage[T1]{fontenc}
\usepackage[latin1]{inputenc}
\usepackage{typearea}
\usepackage{geometry}
\usepackage{ulem}
\usepackage{amsmath}
\usepackage{amsfonts}
\usepackage{amssymb}
\usepackage{latexsym}
\usepackage{enumerate}
\usepackage{verbatim}
\usepackage{amsthm}
\usepackage[all]{xy}
\usepackage{hhline}
\usepackage{epsf} %fuer bildchen
\usepackage{cite}
\usepackage{mathrsfs}
\numberwithin{equation}{section}

\usepackage{color}
\usepackage{dsfont}

\renewcommand{\1}{\mathds{1}}

\newcommand{\Wbar}{\overline{W}}

\newtheorem{theorem}{{\sc Theorem}}[section]
\newtheorem{lemma}[theorem]{{\sc Lemma}}
\newtheorem{prop}[theorem]{{\sc Proposition}}

\theoremstyle{remark}
\newtheorem{remark}[theorem]{{\sc Remark}}

\theoremstyle{definition}

\newcommand{\R}{\mathbb{R} }

\newcommand{\mfk}{\mathfrak{f}}

\newcommand{\N}{\mathbb{N} }

\newcommand{\F}{\mathcal{F}}

\newcommand{\D}{\mathcal{D}}
\newcommand{\W}{\mathcal{W}}

\newcommand{\calZ}{\mathcal{Z}}
\newcommand{\scrZ}{\mathscr{Z}}

\newcommand{\Prob}{\mathbb{P}}
\newcommand{\E}{\mathbb{E}}

\providecommand{\abs}[1]{\lvert #1\rvert}
\providecommand{\babs}[1]{\bigl\lvert #1\bigr\rvert}

\providecommand{\norm}[1]{\lVert #1\rVert}

\providecommand{\Enorm}[1]{\lVert #1\rVert_2}

\DeclareMathOperator{\Var}{Var}

\DeclareMathOperator{\dom}{dom}

\renewcommand{\phi}{\varphi}
\renewcommand{\epsilon}{\varepsilon}

\renewcommand{\rho}{\varrho}
\renewcommand{\P}{\mathbb{P}}

\begin{document}

\title[Fourth moment theorems and product formulae]{ Fourth moment theorems on the Poisson space:\\ analytic statements via product formulae  }

\author{  Christian D\"obler \and Giovanni Peccati}
\thanks{\noindent Universit\'{e} du Luxembourg, Unit\'{e} de Recherche en Math\'{e}matiques \\
E-mails: christian.doebler@uni.lu, giovanni.peccati@gmail.com\\
{\it Keywords}: Multiple Wiener-It\^{o} integrals; Poisson functionals; product formula; fourth moment theorem; carr\'{e}-du-champ operator; Berry-Esseen bounds; Gaussian approximation; Malliavin calculus; Stein's method\\
{\it AMS 2000 Classification}: 60F05; 60H07; 60H05}
\begin{abstract} { We prove necessary and sufficient conditions for the asymptotic normality of multiple integrals with respect to a Poisson measure on a general measure space, expressed both in terms of norms of contraction kernels and of variances of {\it carr\'e-du-champ operators}. Our results substantially complete the fourth moment theorems recently obtained by D\"obler and Peccati (2018) and D\"obler, Vidotto and Zheng (2018). An important tool for achieving our goals is a novel product formula for multiple integrals under minimal conditions.}

\end{abstract} 

\maketitle

\section{Introduction}\label{intro}
\subsection{Overview} {A ``fourth moment theorem'' (FMT) is a probabilistic statement, implying that a certain sequence of (centred and normalised) random variables verifies a central limit theorem (CLT) as soon as the sequence of their fourth moments converges to 3, that is, to the fourth moment of the one-dimensional standard Gaussian distribution. Well-known examples of FMTs are {\it de Jong-type theorems} for degenerate $U$-statistics \cite{DP16}, as well as the class of CLTs for multiple stochastic integrals with respect to Gaussian fields discussed e.g. in \cite[Chapter 5]{NouPecbook}. The aim of the present paper is to provide an analytical study of FMTs for sequences of multiple stochastic integrals with respect to a general Poisson random measure.}

Our starting point will be the recent paper \cite{DP18a}, where we proved quantitative FMTs for multiple Wiener-It\^{o} integrals 
with respect to a general Poisson random measure {under mild regularity assumptions for the involved random variables}. Such regularity assumptions were later removed in \cite{DVZ18}, where the above mentioned fourth moment results were also extended to the multivariate case. 
\smallskip

The common scheme in the proofs of the main results in both papers \cite{DP18a, DVZ18} was { the implementation} of the so-called \textit{spectral viewpoint}, initiated in the seminal paper \cite{Led12} in the framework of invariant measures of a Markov diffusion generator $L$. { In particular, the two references \cite{DP18a, DVZ18} amply demonstrated that the use of \textit{carr\'e-du-champ operators} continues to be the right tool for deriving effective quantitative fourth moment bounds, even when dealing with the {\it non-diffusive} Ornstein-Uhlenbeck generator $L$ defined on configuration spaces (see the discussion below for precise definitions).}

\smallskip

{ Remarkably, the use of {\it carr\'e-du-champ} operators also entails} that neither the results of \cite{DP18a} nor those of \cite{DVZ18} rely on the classical product formula for multiple integrals on the Poisson space (see e.g. \cite[Proposition 5]{Lastsv}), whereas product formulae on Gaussian spaces are pivotal in the original proof of the FMT for multiple integrals with respect to an isonormal Gaussian process -- see e.g. in \cite[Chapter 5]{NouPecbook}. { In particular, neither the proofs nor the statements contained in \cite{DP18a, DVZ18} exploited the properties of \textit{contraction operators}, naturally emerging in connection with multiple stochastic integrals whenever one deals with the computation of higher moments. }

\smallskip

{ It is important to notice that, in the Gaussian situation (see e.g. 
\cite[Theorem 5.2.7]{NouPecbook} and the discussion therein) and differently from what is currently known on the Poisson space, the conditions for the asymptotic normality of a sequence of multiple integrals belonging to a fixed chaos can be phrased equivalently in terms of fourth moments, of norms of contraction kernels, and of variances of squared norms of Malliavin derivatives. As one can deduce by inspecting the hundreds of references listed on the webpage \cite{Web}, conditions expressed in terms of contraction kernels are very often (and by far) the easiest to check and the most amenable to analysis, whereas the exact computation of fourth moments conditions often involves a number of combinatorial difficulties.}
\smallskip

The main objective of the present paper is to complement the fourth moment conditions from \cite{DP18a, DVZ18} with conditions expressed in terms of norms of contractions and in terms of variances of carr\'e du champ operators, { with the precise aim of making the results of \cite{DP18a, DVZ18} even more amenable to analysis}. Under mild { uniform} integrability requirements, we will show that all these conditions are equivalent to the asymptotic normality of the sequence of multiple integrals. Our main findings appear below in Theorem \ref{fmt} (one-dimensional case) and Theorem \ref{mfmt} (multi-dimensional case). 

\smallskip

In order to establish our results, we will need to review the product formula for multiple Wiener- It\^{o} integrals on the Poisson space. In its standard form (see e.g. \cite[Proposition 5]{Lastsv}), such a formula needs additional $L^2$-integrability assumptions on some family of contraction kernels. { In order to be able to work in full generality, we will prove below a new general product formula that, under the { minimal} assumptions of the FMT of \cite{DVZ18}, ensures that certain linear combinations of contraction kernels are automatically in the corresponding $L^2$ spaces. On the other hand, by proving also a converse statement, we obtain a new test criterion, ensuring that the product of two multiple integrals is in $L^2$. Such a criterion is only expressed in terms of linear combinations of contractions kernels. All our new findings on this matter are contained in Theorem \ref{prodform}. }

\smallskip

For the rest of the paper, every random object is defined on a common probability space $(\Omega, \mathscr{F}, \P)$, with $\E$ denoting expectation with respect to $\P$. We use the symbol $\stackrel{\D}{\longrightarrow}$ in order to indicate convergence in distribution of random variables.

\section{A new general product formula on the Poisson space}
We start by providing a rough description of the technical setup, that is needed in order to state and understand our main results. 
We refer to Section \ref{s:generalpoisson}, as well as to \cite{Lastsv}, for precise definitions and more detailed discussions.
\smallskip

Let us fix an arbitrary measurable space $(\mathcal{Z},\mathscr{Z})$ endowed with a $\sigma$-finite measure $\mu$. Furthermore, we let 
\[\mathscr{Z}_\mu:=\{B\in\mathscr{Z}\,:\,\mu(B)<\infty\}\]
and denote by 
\begin{equation*}
\eta=\{\eta(B)\,:\,B\in{ \mathscr{Z}}\}
\end{equation*}
a \textit{Poisson random measure} on $(\mathcal{Z},\mathscr{Z})$ with \textit{control} $\mu$, defined on a suitable probability space $(\Omega,\F,\Prob)$. We recall that this means that: (i) for each finite sequence $B_1,\dotsc,B_m\in\mathscr{Z}$ of pairwise disjoint sets, the random variables $\eta(B_1),\dotsc,\eta(B_m)$ are independent, and (ii) that for every $B\in\mathscr{Z}$, the random variable $\eta(B)$ has the Poisson distribution with mean $\mu(B)$.
Here, we have extended the family of Poisson distributions to the parameter region $[0,+\infty]$ in the usual way. 
For $B\in\mathscr{Z}_\mu$, we also write 
$\hat{\eta}(B):=\eta(B)-\mu(B)$ and denote by 
\[\hat{\eta}=\{\hat{\eta}(B)\,:\,B\in\mathscr{Z}_\mu\}\]
the \textit{compensated Poisson measure} associated with $\eta$. Throughout this paper we will assume that $\F=\sigma(\eta)$. In order to state our main results, we have to briefly recall the following objects, { arising in the context of stochastic analysis for Poisson measures}. By $L$, we denote the generator of the \textit{Ornstein-Uhlenbeck semigroup} associated with $\eta$, and by $\dom L\subseteq L^2(\P)$ we denote its domain. It is well-known that $-L$ is a symmetric, diagonalizable operator on $L^2(\P)$ which has has the pure point spectrum 
$\N_0=\{0,1,\dotsc\}$. Closely connected to $L$ is the symmetric, bilinear and nonnegative \textit{carr\'e du champ operator} $\Gamma$, which is defined by 
\[\Gamma(F,G)=\frac12\Bigl(L(FG)-FLG-GLF\Bigr)\,,\]
for all $F,G\in \dom L$ such that also $FG\in\dom L$.
 For $p\in\N_0$ we denote by $C_p:=\ker(L+p{\rm Id})$ the so-called \textit{$p$-th Wiener chaos} { associated with }$\eta$. Here, we denote by ${\rm Id}$ the identity operator on $L^2(\P)$. It is a well-known fact that, for $p\in\N$, the linear space $C_p$ coincides with the collection of \textit{multiple Wiener-It\^{o} integrals} $I_p(f)$ of order $p$ with respect to ${ \hat{\eta}}$. Here, $f\in L^2(\mu^p)$ is a square-integrable function on the product space 
$(\mathcal{Z}^p,\mathscr{Z}^{\otimes p}, \mu^p)$. Moreover, for a constant 
$c\in\R$ we let $I_0(c):=c$. Then, since the kernel of $L$ coincides with the constant random variables, we also have 
$C_0=\{I_0(c)\,:\, c\in\R\}$. Multiple integrals have the following two fundamental properties.
Let $p,q\geq0$ be integers: then,
\begin{enumerate}[1)]
 \item $I_p(f)=I_p(\tilde{f})$, where $\tilde{f}$ denotes the \textit{canonical symmetrization} of $f\in L^2(\mu^p)$, i.e., with $\mathbb{S}_p$ the symmetric group acting on $\{1,\dotsc,p\}$, we have
 \[\tilde{f}(z_1,\dotsc,z_p)=\frac{1}{p!}\sum_{\pi\in\mathbb{S}_p} f(z_{\pi(1)},\dotsc,z_{\pi(p)})\,;\]
 \item $I_p(f)\in L^2(\Prob)$, and $\E\bigl[I_p(f)I_q(g)\bigr]= \delta_{p,q}\,p!\,\langle \tilde{f},\tilde{g}\rangle $, where $\delta_{p,q}$ denotes {Kronecker's delta symbol}.
\end{enumerate}
{ The Hilbert subspace of $L^2(\mu^p)$ composed of $\mu^p$-a.e. symmetric kernes will henceforth be denoted by $L_s^2(\mu^p)$}. It is a crucial fact that every $F\in L^2(\Prob)$ admits 
a unique representation 
\begin{equation}\label{chaosdec}
 F=\E[F]+\sum_{p=1}^\infty I_p(f_p)\,,
\end{equation}
where $f_p\in L_s^2(\mu^p)$, $p\geq1$, are suitable symmetric kernel functions, and the series converges in $L^2(\Prob)$. Identity \eqref{chaosdec} is referred to as the \textit{chaotic decomposition} of the functional $F\in L^2(\Prob)$. Hence, 
multiple integrals are in a way the basic building blocks of the space $L^2(\P)$. Note that \eqref{chaosdec} can equivalently be written as 
\begin{equation*}
L^2(\P)=\bigoplus_{p=0}^\infty C_p\,,
\end{equation*}
where the sum on the right hand side is furthermore orthogonal.

\smallskip

The following analytic notion of a \textit{contraction kernel} will also be crucial for the statements of our results. Fix integers $p,q\geq 1$ as well as symmetric kernels $f\in L_s^2(\mu^p)$ and $g\in L_s^2(\mu^q)$. For integers 
$1\leq l\leq r\leq p\wedge q$ we define the \textit{contraction kernel} $f\star_l^r g$ on $\mathcal{Z}^{p+q-r-l}$ by 
\begin{align*}
 &(f\star^l_r g)(y_1,\dotsc,y_{r-l},t_1,\dotsc,t_{p-r}, s_1,\dotsc,s_{q-r})\\
 &:=\int_{\mathcal{Z}^l}\Bigl( f\bigl(x_1,\dotsc,x_l, y_1,\dotsc,y_{r-l},t_1,\dotsc,t_{p-r}\bigr)\\
 &\hspace{3cm}\cdot g\bigl(x_1,\dotsc,x_l,y_1,\dotsc,y_{r-l},s_1,\dotsc,s_{q-r}\bigr)\Bigr)d\mu^l(x_1,\dotsc,x_l)\,.
\end{align*}
If $l=0$ and $0\leq r\leq p\wedge q$, then we let 
\begin{align*}
 &(f\star^0_r g)(y_1,\dotsc,y_{r},t_1,\dotsc,t_{p-r}, s_1,\dotsc,s_{q-r})\\
 &:=f(y_1,\dotsc,y_{r},t_1,\dotsc,t_{p-r})\cdot g(y_1,\dotsc,y_{r}, s_1,\dotsc,s_{q-r})\,.
\end{align*}
In particular, if $l=r=0$, then $f\star_0^0g=f\otimes g$ reduces to the \textit{tensor product} of $f$ and $g$, given by 
\begin{equation*}
 (f\otimes g)(t_1,\dotsc,t_p,s_1,\dotsc,s_q):=f(t_1,\dotsc,t_p)\cdot g(s_1,\dotsc,s_q)\,.
\end{equation*}
More generally, whenever $0\leq l=r\leq p\wedge q$, then on customarily writes $f\otimes_r g$ for $f\star_r^r g$. In this case, a simple application of the Cauchy-Schwarz inequality shows that $f\otimes_r g\in L^2(\mu^{p+q-2r})$. This however, does not hold for 
$l< r$, in general. For instance, we have $f\star_p^0 f=f^2$, which is in $L^2(\mu^{p})$ if and only if $f\in L^4(\mu^p)$.  %in general, $f\star_l^r g$ is not in $L^2(\mu^{p+q-r-l})$ even though $f\in L^2(\mu^p)$ and $g\in L^2(\mu^q)$. 
As shown in our paper \cite{DP18b},
the contraction kernel $f\star_l^r g$ is always $\mu^{p+q-r-l}$-a.e. well-defined as a function on $\mathcal{Z}^{p+q-r-l}$.

Contraction kernels play a major role in this article because they naturally arise in the following classical product formula on the Poisson space that is taken from \cite{Lastsv}. It was first proved under less general conditions in \cite{Sur}.

\begin{prop}[Classical product formula]\label{prodformcl}
 Let $p,q\geq 1$ be integers and assume that $f\in L^2_s(\mu^p)$ and $g\in L^2_s(\mu^q)$. If, for all integers $0\leq r\leq p\wedge q$ and $0\leq l\leq r-1$ one has that $f\star_r^l g\in L^2(\mu^{p+q-r-l})$, then 
 \begin{equation}\label{pfcl}
  I_p(f) I_q(g)=\sum_{r=0}^{p\wedge q} r!\binom{p}{r}\binom{q}{r}\sum_{l=0}^r\binom{r}{l} I_{p+q-r-l}\bigl(\widetilde{f\star_r^l g}\bigr)\,.
 \end{equation}
\end{prop}

Note that the sum appearing on the right hand side of \eqref{pfcl} is not orthogonal, since it is not arranged according to the orders of the integrals. Introducing the parameter $m=r+l$, satisfying the restrictions 
\[0\leq r\leq m\leq2r\leq2(p\wedge q),\]
we can rewrite \eqref{pfcl} as follows:
\begin{align}\label{pfcl2}
 I_p(f) I_q(g)&=\sum_{r=0}^{p\wedge q} r!\binom{p}{r}\binom{q}{r}\sum_{m=r}^{2r}\binom{r}{m-r} I_{p+q-m}\bigl(\widetilde{f\star_r^{m-r} g}\bigr)\notag\\
 &=\sum_{m=0}^{2(p\wedge q)}\sum_{r=\lceil\frac{m}{2}\rceil}^{m\wedge p\wedge q}r!\binom{p}{r}\binom{q}{r}\binom{r}{m-r}\bigl(\widetilde{f\star_r^{m-r} g}\bigr)\notag\\
 &=\sum_{m=0}^{2(p\wedge q)} I_{p+q-m}\bigl(h_{p+q-m}\bigr)\,,
\end{align}
where the symmetric kernel $h_{p+q-m}$, $0\leq m\leq 2(p\wedge q)$, is given by 
\begin{align}\label{hpqm}
 h_{p+q-m}&=\sum_{r=\lceil\frac{m}{2}\rceil}^{m\wedge p\wedge q}r!\binom{p}{r}\binom{q}{r}\binom{r}{m-r}\bigl(\widetilde{f\star_r^{m-r} g}\bigr)\notag\\
&=\sum_{r=\lceil\frac{m}{2}\rceil}^{m\wedge p\wedge q}\frac{p!q!}{(p-r)!(q-r)!(m-r)!(2r-m)!}\bigl(\widetilde{f\star_r^{m-r} g}\bigr)\,.
\end{align}

From the classical product formula given in Theorem \ref{prodformcl} we can conclude the validity of \eqref{pfcl2} only under the assumption that all the contraction kernels $f\star_l^{r} g$, $0\leq l<r\leq p\wedge q$, be in $L^2$ which implies that also the $h_{p+q-m}$ are in $L^2$. Note also that, under such a restriction, the sum in \eqref{pfcl2} is orthogonal.
 
 \smallskip
 
The first theoretical result of this paper is a new general product formula, showing that \eqref{pfcl} continues indeed to hold under the minimal assumption that $I_p(f) I_q(g)\in L^2(\P)$. In particular, this implies that $h_{p+q-m}$ as given in \eqref{hpqm} is always in 
$L^2(\mu^{p+q-m})$ even though this might not be the case for the individual contractions $f\star_r^{m-r} g$ appearing in the defining equation \eqref{hpqm}. Moreover, we also show that the converse is true as well, i.e. that $I_p(f) I_q(g)\in L^2(\P)$ whenever each kernel $h_{p+q-m}$ is in $L^2(\mu^{p+q-m})$. 

\begin{theorem}[General product formula on the Poisson space]\label{prodform}
 Suppose that $p,q\geq 1$ are integers, select $f\in L^2_s(\mu^p)$ and $g\in L^2_s(\mu^q)$, and define $F:=I_p(f)$, $G:=I_q(g)$. Then, the product $FG$ is in $L^2(\P)$ if and only if, for each $0\leq m\leq 2(p\wedge q)$, the kernel $h_{p+q-m}$ given by \eqref{hpqm} is in $L^2(\mu^{p+q-m})$. In this case, $FG$ has the finite chaotic decomposition 
 \begin{equation}\label{pf1}
  FG=\sum_{m=0}^{2(p\wedge q)} I_{p+q-m}(h_{p+q-m})\,.
 \end{equation}
\end{theorem}

\begin{remark}
\begin{enumerate}[(a)]
\item An immediate consequence of Theorem \ref{prodform} is that the product of two multiple integrals with respect to { $\hat{\eta}$} either has a finite chaos decomposition of order at most $p+q$ or none at all. To the best of our knowledge this fact has not been noted so far.
\item Identity \eqref{pf1} in particular holds, whenever $F,G\in L^4(\P)$.
\item A similar product formula as in Theorem \ref{prodform}, but under less general conditions, can be found in \cite[Chapter 6]{Priv}.
\end{enumerate}
\end{remark}

\section{An extension of the fourth moment theorem}
In the recent paper \cite{DP18a}, under mild integrability conditions, we proved the bounds 
\begin{align}
d_\W(F,N)&\leq \left(\sqrt{\frac{2}{\pi}}+2\right)\sqrt{\E[F^4]-3}\label{fmw}\,,\\
d_\mathcal{K}(F,N)&\leq 15.6\sqrt{\E[F^4]-3}\label{fmk}\,,
\end{align} 
where $F=I_p(f)\in C_p$ is a multiple Wiener-It\^o integral with respect to $\eta$ such that $\E[F^2]=1$, $N$ is a standard normal random variable and $d_\W$ and $d_\mathcal{K}$ denote the Wasserstein and Kolmogorov distances, respectively { (see e.g. \cite[Appendix C]{NouPecbook} and the references therein)}. The above mentioned integrability conditions could be successfully removed for the Wasserstein bound \eqref{fmw} in the paper \cite{DVZ18}. In particular, one has the following sequential FMT: {\it Suppose that, for each $n\in\N$, 
$p_n\geq1$ is an integer, $f_n\in L_s^2(\mu^{p_n})$ is a kernel and $F_n=I_{p_n}(f_n)$ is a multiple integral such that 
\[\lim_{n\to\infty}\E[F_n^2]=\lim_{n\to\infty}p_n!\norm{f_n}_2^2=1\quad\text{and}\quad \lim_{n\to\infty}\E[F_n^4]=3\,; \]
then, $(F_n)_{n\in\N}$ converges in distribution to the standard normal random variable $N$}. 

\smallskip

{ In what follows, we will state and prove a substantial refinement of such a result, in the crucial case of a sequence of multiple stochastic integrals belonging to a fixed chaos, that is, such that $p_n\equiv p\geq1$.} In order to do this, for $n\geq 1$, we need to define the auxiliary kernels 
\begin{align}
 h_{2p-m}^{(n)}&=\sum_{s=\lceil\frac{m}{2}\rceil}^{m\wedge p} \frac{(p!)^2}{\bigl((p-s)!\bigr)^2 (2s-m)!(m-s)!}\widetilde{\bigl(f_n\star_{s}^{m-s}f_n\bigr)}\,,\quad 0\leq m\leq { 2}p\,.\label{e:hnpqm}
\end{align}

{ The next statement is one of the main achievements of the paper.}

\begin{theorem}[Extended fourth moment theorem]\label{fmt}
Fix an integer $p\geq1$ and let $F_n=I_p(f_n)$, $n\in\N$, be a sequence in $C_p$ such that $\lim_{n\to\infty}\E[F_n^2]=1$, with $f_n\in L_s^2(\mu^p)$ for $n\in\N$. Let $N\sim N(0,1)$ be a standard normal random variable. Consider the following conditions:
\begin{enumerate}[{\normalfont(i)}]
\item $F_n\stackrel{\D}{\longrightarrow}N$ as $n\to\infty$.
\item $\lim_{n\to\infty}\E[F_n^4]=3$.
\item $\lim_{n\to\infty}\norm{h^{(n)}_{2p-m}}_2=0$ for all $m=1,\dotsc, 2p-1$ and 
$\lim_{n\to\infty} \norm{f_n\otimes_r f_n}_2=0$ for all $r=1,\dotsc,p-1$. 
\item { $\E[F_n^4]<\infty$ for $n$ large enough}, and $\Gamma(F_n,F_n)\stackrel{L^2(\P)}{\longrightarrow} p$ (or, equivalently, \\
$\lim_{n\to\infty}\Var\bigl(\Gamma(F_n,F_n)\bigr)=0)$.
\item $\lim_{n\to\infty}\norm{h^{(n)}_{2p-m}}_2=0$ for all $m=1,\dotsc, 2p-1$.
\end{enumerate}
Then, we have { the implications}
\[{\normalfont\text{(iii)}\Leftrightarrow\text{(ii)}\Leftrightarrow\text{(iv)}\Leftrightarrow\text{(v)}\Rightarrow\text{(i)}}\,.\]
Moreover, if the sequence $(F_n^4)_{n\in\N}$ is uniformly integrable, then the conditions {\normalfont(i)-(v)} are equivalent.
\end{theorem}

{\begin{remark}
It is very interesting and quite surprising that condition (iii) and the seemingly weaker condition (v) in Theorem \ref{fmt} are indeed equivalent.
\end{remark}}

\begin{proof}[Proof of Theorem \ref{fmt}]
(ii)$\Leftrightarrow$(iii): Using the content of Theorem \ref{prodform}, a straightforward generalization of equation (43) from \cite{DP18a} and identity (5.2.12) from the book \cite{NouPecbook} yield 
\begin{align}
 &\E[F_n^4]-\E[F_n^2]^2=\sum_{k=1}^{2p-1}k!\norm{h_k^{(n)}}_2^2+ (2p)!\norm{f_n\tilde{\otimes} f_n}_2^2\label{fm0}\\
 &=\sum_{m=1}^{2p-1}(2p-m)!\norm{h_{2p-m}^{(n)}}_2^2 +2(p!)^2\norm{f_n}_2^4+(p!)^2\sum_{r=1}^{p-1}\binom{p}{r}^2\norm{f_n\otimes_r f_n}_2^2\notag\\
 &=2(p!)^2\norm{f_n}_2^4+\sum_{m=1}^{2p-1}(2p-m)!\norm{h_{2p-m}^{(n)}}_2^2 +(p!)^2\sum_{r=1}^{p-1}\binom{p}{r}^2\norm{f_n\otimes_r f_n}_2^2\label{fm1}\,,
\end{align}
where $h_{2p-m}^{(n)}$ { is defined in \eqref{e:hnpqm}.}
%\begin{align*}
% h_{2p-m}^{(n)}&=\sum_{s=\lceil\frac{m}{2}\rceil}^{m\wedge p} \frac{(p!)^2}{\bigl((p-s)!\bigr)^2 (2s-m)!(m-s)!}\widetilde{\bigl(f_n\star_{s}^{m-s}f_n\bigr)}\,,\quad 0\leq m\leq 2p\,.
%\end{align*}
Now, observing that, by assumption, 
\begin{equation*}
\lim_{n\to\infty}\E[F_n^2]=\lim_{n\to\infty}p!\norm{f_n}_2=1\,,
\end{equation*}
we conclude from \eqref{fm1} that (ii) and (iii) are indeed equivalent.

\smallskip

\noindent(ii)$\Rightarrow$(iv): By Lemma 3.1 of \cite{DP18a} we have the inequality 
\begin{equation*}
\frac{(2p-1)^2}{4p^2}\bigl(\E[F_n^4]-3\E[F_n^2]^2\bigr)\geq \Var\bigl(p^{-1}\Gamma(F_n,F_n)\bigr)\,,
\end{equation*}
which immediately implies the claim. 

\smallskip

\noindent(iv)$\Rightarrow$(ii): By Part 2 of \cite[Remark 5.2]{DVZ18} { (that we can apply, since (iv) ensures that $\E[F_n^4]<\infty$ for large $n$)} we have the bound 
\begin{equation*}
\E[F_n^4]-3\E[F_n^2]^2\leq\frac{6}{p}\Var\bigl(\Gamma(F_n,F_n)\bigr)\,,
\end{equation*}
proving the implication.
\smallskip

\noindent(iv)$\Rightarrow$(v): Using Theorem \ref{prodform}, from the computations on page 1895 in \cite{DP18a} { (that we can apply since, under (iv), one has that $\E[F_n^4]<\infty$ for large $n$)}, we have 
\begin{equation}\label{e:glo}
\Var\bigl(p^{-1}\Gamma(F_n,F_n)\bigr)=\frac{1}{4p^2}\sum_{m=1}^{2p-1}m^2(2p-m)!\norm{h_{2p-m}^{(n)}}_2^2\,,
\end{equation}
and the desired implication follows.

{ \noindent(v)$\Rightarrow$(iv): Using again Theorem \ref{prodform}, we see that, under (v), $F_n\in L^4(\P)$ for $n$ large enough, so that we can apply once again \eqref{e:glo} to deduce the claim.}
\smallskip

\noindent(ii)$\Rightarrow$(i): This is an immediate consequence of the fourth moment bound \eqref{fmw}.

\smallskip
 
\noindent Finally, assume that the sequence $(F_n^4)_{n\in\N}$ is uniformly integrable and that (i) holds. Then, we have 
\[\lim_{n\to\infty}\E[F_n^4]=\E[N^4]=3\]
so that condition (ii) is satisfied. This concludes the proof.
\end{proof}

\subsection{Multivariate extended fourth moment theorems}
In \cite[Corollary 1.8]{DVZ18} the following Peccati-Tudor type theorem on the Poisson space was proved: 
\begin{prop}\label{DVZ}
Fix a dimension $d\in\N$ as well as positive integers $p_1,\dotsc,p_d$. Moreover, for each $n\in\N$, suppose that 
$F^{(n)}:=(F_1^{(n)},\dotsc,F_d^{(n)})^T$ is a random vector such that $F_k^{(d)}=I_{p_k}(f_k^{(n)})\in C_{p_k}$, where 
$f_k^{(n)}\in L_s^2(\mu^{p_k})$, $k=1,\dotsc,d$. Furthermore, suppose that $V=(V(i,j))_{1\leq i,j\leq d}$ is a nonnegative definite matrix such that the covariance matrix $\Sigma_n$ of $F^{(n)}$ converges to $V$ as $n\to\infty$. 
Also, suppose that $N=(N_1,\dotsc,N_d)^T$ is a $d$-dimensional centered Gaussian vector with covariance matrix $V$. If $\lim_{n\to\infty}\E\bigl[(F_k^{(n)})^4\bigr]=3 V(k,k)^2$ for all $1\leq k\leq d$, then, as $n\to\infty$, the random vector $F^{(n)}$ converges in distribution to $N$.  
\end{prop}

Under the assumptions of Proposition \ref{DVZ}, for $n\in\N$ and $k=1,\dotsc,d$, define the symmetric kernels $h_{k,2p_k-m}^{(n)}$, $0\leq m\leq  2p_k$, by 
\begin{align*}
 h_{k,2p_k-m}^{(n)}&=\sum_{s=\lceil\frac{m}{2}\rceil}^{m\wedge p_k} \frac{(p_k!)^2}{\bigl((p_k-s)!\bigr)^2 (2s-m)!(m-s)!}\widetilde{\bigl(f_n\star_{s}^{m-s}f_n\bigr)}\,.
\end{align*}

With this notation at hand we can state our multivariate extended fourth moment theorem: 
\begin{theorem}[Extended Multivariate FMT]\label{mfmt}
Fix a dimension $d\in\N$ as well as positive integers $p_1,\dotsc,p_d$. Moreover, for each $n\in\N$, suppose that 
$F^{(n)}:=(F_1^{(n)},\dotsc,F_d^{(n)})^T$ is a random vector such that $F_k^{(d)}=I_{p_k}(f_k^{(n)})\in C_{p_k}$, where 
$f_k^{(n)}\in L_s^2(\mu^{p_k})$, $k=1,\dotsc,d$. Furthermore, suppose that $V=(V(i,j))_{1\leq i,j\leq d}$ is a nonnegative definite matrix such that the covariance matrix $\Sigma_n$ of $F^{(n)}$ converges to $V$ as $n\to\infty$. 
Suppose that $N=(N_1,\dotsc,N_d)^T$ is a $d$-dimensional centered Gaussian vector with covariance matrix $V$. 
Consider the following conditions: 
\begin{enumerate}[{\normalfont (i)}]
\item $F^{(n)}\stackrel{\D}{\longrightarrow}N$ as $n\to\infty$.
\item For all $k=1,\dotsc,d$: $\lim_{n\to\infty}\E\bigl[(F_k^{(n)})^4\bigr]=3V(k,k)^2$.
\item For all $k=1,\dotsc,d$: $\lim_{n\to\infty}\norm{h^{(n)}_{k,2p_k-m}}_2=0$ for all $m=1,\dotsc, 2p_k-1$ and 
$\lim_{n\to\infty} \norm{f_k^{(n)}\otimes_r f_k^{(n)}}_2=0$ for all $r=1,\dotsc,p_k-1$. 
\item { For $n$ large enough, $F_k^{(n)}\in L^4(\P)$ for every $k=1,...,d$} and, again for all $k=1,\dotsc,d$, $\Gamma(F_k^{(n)},F_k^{(n)})\stackrel{L^2(\P)}{\longrightarrow} pV(k,k)$\\
 (or, equivalently, $\lim_{n\to\infty}\Var\bigl(\Gamma(F_k^{(n)},F_k^{(n)})\bigr)=0)$.
\item For all $k=1,\dotsc,d$: $\lim_{n\to\infty}\norm{h^{(n)}_{k,2p_k-m}}_2=0$ for all $m=1,\dotsc, 2p_k-1$.
\end{enumerate}
Then, we have the implications
\[{\normalfont\text{(iii)}\Leftrightarrow\text{(ii)}\Leftrightarrow\text{(iv)}\Leftrightarrow\text{(v)}\Rightarrow\text{(i)}}\,;\]
moreover if, for each $k=1,\dotsc,d$, the sequence $\bigl((F_k^{(n)})^4\bigr)_{n\in\N}$ is uniformly integrable, then all the conditions {\normalfont(i)-(v)} are equivalent. 
\end{theorem}

\begin{proof}
The equivalence of items (ii)-(v) can be proved similarly as in the proof of Theorem \ref{fmt}. 
One only has to extend the arguments to general positive variances of the coordinates of $N$. That (ii) implies (i) follows from Theorem \ref{DVZ}. Finally, the fact that (i) implies (ii) under the assumption of uniform integrability follows as in the previous proof.
\end{proof}

\section{General Poisson point processes and technical framework}\label{s:generalpoisson}
Here we describe our theoretical framework by adopting the language of \cite{Lastsv} (see also \cite{LPbook}).

\smallskip

Let $(\mathcal{Z},\mathscr{Z})$ be an arbitrary measurable space endowed with a $\sigma$-finite measure $\mu$ and denote by $\mathbf{N}_\sigma=\mathbf{N}_\sigma(\mathcal{Z})$ the space of all $\sigma$-finite point measures $\chi$ on $(\mathcal{Z},\mathscr{Z})$ that satisfy $\chi(B)\in\N_0\cup\{+\infty\}$ for all $B\in\mathscr{Z}$. This space is equipped with the smallest 
$\sigma$-field $\mathscr{N}_\sigma:=\mathscr{N}_\sigma(\calZ)$ such that, for each $B\in\scrZ$, the mapping $\mathbf{N}_\sigma\ni\chi\mapsto\chi(B)\in[0,+\infty]$ is measurable. It is convenient to view the Poisson process $\eta$ as a random element 
of the measurable space $(\mathbf{N}_\sigma,\mathscr{N}_\sigma)$, defined on an abstract probability space $(\Omega, \F,\P)$. Without loss of generality we may assume that $\F=\sigma(\eta)$.
Moreover, we denote by $\mathbf{F}(\mathbf{N}_\sigma)$ the class of all measurable functions $\mathfrak{f}:\mathbf{N}_\sigma\rightarrow\R$ and by $\mathcal{L}^0(\Omega):=\mathcal{L}^0(\Omega,\F)$ the class of real-valued, measurable functions $F$ on $\Omega$. 
Note that, as $\F=\sigma(\eta)$, each $F\in \mathcal{L}^0(\Omega)$ can be written as $F=\mfk(\eta)$ for some measurable function $\mfk$. This $\mfk$, called a {\it representative} of $F$, is $\Prob_\eta$-a.s. uniquely defined, where $\Prob_\eta=\Prob\circ\eta^{-1}$ is the image measure of $\Prob$ under $\eta$ on the space $(\mathbf{N}_\sigma,\mathscr{N}_\sigma)$. For  $F=\mfk(\eta)\in\mathcal{L}^0(\Omega)$ and $z\in\mathcal{Z}$ we define the \textit{add one cost operators} $D_z^+$, $z\in\mathcal{Z}$, by 
\begin{equation*}
D_z^{+} F:=\mfk(\eta+\delta_z)-\mfk(\eta)\,.
\end{equation*}
It is straightforward to verify the following product rule: For $F,G\in \mathcal{L}^0(\Omega)$ and $z\in\mathcal{Z}$ one has
\begin{align}\label{prodrule}
D^+_z(FG)&=GD^+_zF+FD_z^+G+D_z^+F D_z^+G\,.
\end{align}

More generally, if $m\in\N$ and $z_1,\dotsc, z_m\in\mathcal{Z}$, then we define inductively  $D_{z_1}^{(1)}=D_{z_1}^+$ and
\begin{equation*}
D_{z_1,\dotsc,z_m}^{(m)} F:=D_{z_1}^+\bigl(D_{z_2,\dotsc,z_m}^{(m-1)}F\bigr)\,,\quad m\geq2\,.
\end{equation*}
It is easily seen that 
\begin{equation}\label{Dm}
D_{z_1,\dotsc,z_m}^{(m)} F =\sum_{J\subseteq[m]}(-1)^{m-\abs{J}}\; \mfk\Bigl(\eta+\sum_{i\in J}\delta_{z_i}\Bigl)
\end{equation}
which shows that the mapping $\Omega\times\mathcal{Z}^m\ni (\omega,z_1,\dotsc,z_m)\mapsto D_{z_1,\dotsc,z_m}^{(m)} F(\omega)
\in\R$ is $\F\otimes\mathscr{Z}^{\otimes m}$-measurable. Moreover, it also implies that $D_{z_1,\dotsc,z_m}^{(m)} F=D_{z_{\sigma(1)},\dotsc,z_{\sigma(m)}}^{(m)} F$ for each permutation $\sigma$ of $[n]$. 
\smallskip

For an integer $p\geq1$ we denote by $L^2(\mu^p)$ the Hilbert space of all square-integrable and real-valued functions on $\mathcal{Z}^p$ and we write $L^2_s(\mu^p)$ 
for the subspace of those functions in $L^2(\mu^p)$ which are $\mu^p$-a.e. symmetric. Moreover, for ease of notation, we denote by $\Enorm{\cdot}$ and $\langle \cdot,\cdot\rangle$ the usual norm and scalar product 
on $L^2(\mu^p)$ for whatever value of $p$. We further define $L^2(\mu^0):=\R$. For $f\in L^2(\mu^p)$, we denote by $I_p(f):=I_p^\eta(f)$ the \textit{multiple Wiener-It\^o integral} of $f$ with respect to $\hat{\eta}$. 
If $p=0$, then, by convention, 
$I_0(c):=c$ for each $c\in\R$.
We refer to Section 3 of \cite{Lastsv} for a precise definition and the following basic properties of these integrals in the general framework of a $\sigma$-finite measure space $(\mathcal{Z},\mathscr{Z},\mu)$. 
If $F=I_p(f)$ for some $p\geq 1$ and $f\in L_s^2(\mu^p)$, then for all $z\in\mathcal{Z}$ one has 
\begin{equation}\label{derint}
 D_z^+F=I_{p-1}\bigl(f(z,\cdot)\bigr)\,.
\end{equation}
In particular, $D_z^+F$ is a multiple Wiener-It\^{o} integral of order $p-1$. If, on the other hand, $p=0$, then it is easy to see that $D^+_z F=0$.

\smallskip

As recalled above, for $p\geq0$ the Hilbert space consisting of all random variables $I_p(f)$, $f\in L^2(\mu^p)$, is called the \textit{$p$-th Wiener chaos} associated with $\eta$, and is customarily denoted by $C_p$. 

\smallskip

 From Theorem 2 in \cite{Lastsv} (which is Theorem 1.3 from the article \cite{LaPen11}) it is known that, for all $F\in L^2(\Prob)$ and all $p\geq1$, the kernel $f_p$ in \eqref{chaosdec} is explicitly given by 
\begin{equation}\label{kerform}
 f_p(z_1,\dotsc,z_p)=\frac{1}{p!}\E\bigl[D^{(p)}_{z_1,\dotsc,z_p}F\bigr]\,,\quad z_1,\dotsc,z_p\in\calZ\,.
\end{equation}
Identity \eqref{kerform} will be an essential tool for the proof of Theorem \ref{prodform}.

\section{Proof of the product formula}

For the proof of Theorem \ref{prodform} we will need the following auxiliary result that 
provides us with a sufficient condition for an integrable random variable $F$ to be in $L^2(\P)$. 
\begin{lemma}\label{l2lemma}
Suppose that $F\in L^1(\P)$ is such that there exists an $M\in\N$ such that 
\begin{enumerate}[{\normalfont (a)}]
\item For all $z_1,\dotsc,z_{M+1}\in\mathcal{Z}$ one has $D^{(M+1)}_{z_1,\dotsc,z_{M+1}}F=0$.
\item For all $m=1,\dotsc,M$ and all $z_1,\dotsc,z_{m}\in\mathcal{Z}$, $D^{(m)}_{z_1,\dotsc,z_{m}}F\in L^1(\P)$\\ and 
$(z_1,\dotsc,z_m)\mapsto\E\bigl[D^{(m)}_{z_1,\dotsc,z_{m}}F\bigr]\in L^2(\mu^m)$.
\end{enumerate}
Then, $F\in L^2(\P)$.
\end{lemma} 

\begin{proof}
The proof relies on the following $L^1$-version of the \textit{Poincar\'{e} inequality} on the Poisson space (see { \cite[Proposition 2.5]{LPS}, as well as} \cite[Corollary 1]{Lastsv}: For $F\in L^1(\P)$ one has 
\begin{equation}\label{poincare}
\E[F^2]\leq (\E[F])^2+\int_\mathcal{Z}\E\bigl[(D_z^{+} F)^2\bigr]\mu(dz)
\end{equation}
with both sides possibly being equal to $+\infty$. The lemma can now be concluded by iterating \eqref{poincare}:
\begin{align*}
\E[F^2]&\leq (\E[F])^2+\int_\mathcal{Z}\E\bigl[(D_{z_1}^{+} F)^2\bigr]\mu(dz_1)\\
&\leq (\E[F])^2+\int_\mathcal{Z}\bigl(\E\bigl[D_{z_1}^{+} F\bigr]\bigr)^2\mu(dz_1) 
+\int_{\mathcal{Z}^2}\E\bigl[\bigl(D_{z_1,z_2}^{(2)} F\bigr)^2\bigr]\mu^2(dz_1,dz_2)\\
&\leq\ldots\leq  (\E[F])^2+\sum_{m=1}^M\int_{\mathcal{Z}^m}\bigl(\E\bigl[D_{z_1,\dotsc,z_m}^{(m)} F\bigr]\bigr)^2\mu^m(dz_1,\dotsc,dz_m)\\
&\;+\int_{\mathcal{Z}^{M+1}}\E\bigl[\bigl(D_{z_1,\dotsc,z_{M+1}}^{(M+1)} F\bigr)^2\bigr]\mu^{M+1}(dz_1,\dotsc,dz_{M+1})\\
&=(\E[F])^2+\sum_{m=1}^M\int_{\mathcal{Z}^m}\bigl(\E\bigl[D_{z_1,\dotsc,z_m}^{(m)} F\bigr]\bigr)^2\mu^m(dz_1,\dotsc,dz_m)
<+\infty\,.
\end{align*}
Note that we have used assumption (a) for the last equality and (b) in order to use \eqref{poincare} iteratively as well as for the final inequality.
\end{proof}

{We now provide a detailed proof of Theorem \ref{prodform}, which has a purely combinatorial nature and does not make any use of recursive arguments.}

\begin{proof}[Proof of Theorem \ref{prodform}] 
We first make the following important observation: whenever $F=I_p(f)$ and $G=I_q(g)$ are as in the statement of the theorem, for all $k\in\N_0$ and all fixed $z_1,\dotsc,z_k\in\mathcal{Z}$ we have that $D^{(k)}_{z_1,\dotsc,z_k}(FG)$ is a (finite) linear combination of products of two multiple Wiener-It\^{o} integrals of orders less than $p$ and $q$, respectively. This easily follows iteratively from \eqref{prodrule} and \eqref{derint}. In particular, all summands appearing in this linear combination and, a fortiori, the quantity $D^{(k)}_{z_1,\dotsc,z_k}(FG)$ itself is in $L^1(\P)$. 
This observation will be used implicitly in the rest of this proof.

Let us now assume first that $FG\in L^2(\P)$. Then, we know that a chaotic decomposition of the form 
\begin{equation*}
FG=\sum_{k=0}^\infty I_k(h_k)
\end{equation*}
exists, with $h_0=\E[FG]$ and $h_k\in L_s^2(\mu^k)$ for each $k\in\N$. In \cite[Lemma 2.4]{DP18a} we already proved 
that $h_k=0$ for all $k\geq p+q$. However, this will also easily follow from the arguments used in the present proof. 
From \eqref{kerform} we immediately get that 
\begin{equation}\label{hk}
h_k(z_1,\dotsc,z_k)=\frac{1}{k!}\E\bigl[D_{z_1,\dotsc,z_k}^{(k)}(FG)\bigr]\,,\quad k\in\N\,, z_1,\dotsc,z_k\in\mathcal{Z}\,.
\end{equation}
In order to get more explicit expressions for the $h_k$ we introduce the following operators: For a pair 
$(X,Y)\in\mathcal{L}^0(\Omega)\times\mathcal{L}^0(\Omega)$ and $z\in\mathcal{Z}$, define
\begin{align*}
D_z^{L}(X,Y)&:=(D_z^{+} X,Y)\,,\notag\\
D_z^{R}(X,Y)&:=(X,D_z^{+}Y)\quad\text{and}\notag\\
D_z^{B}(X,Y)&:=(D_z^{+}X,D_z^{+}Y)\,.
\end{align*}
More generally, if $W=(W_1,\dotsc,W_m)\in\{L,R,B\}^m$ is a \textit{word} of length $\abs{W}=m$ in the \textit{alphabet} 
$\{L,R,B\}$ and $z_1,\dotsc,z_m\in\mathcal{Z}$, then we let 
\[D^{[W_1]}_{z_1}(X,Y):=D_{z_1}^{W_1}(X,Y)\,,\]
if $m=1$ and, for $m\geq2$, we define inductively
\[D^{[W]}_{z_1,\dotsc,z_m}(X,Y):=D_{z_1}^{W_1}\bigl(D^{[W']}_{z_2,\dotsc,z_m}(X,Y)\bigr)\,,\]
where $W'=(W_2,\dotsc,W_m)$.

\smallskip

\noindent Then, with the multiplication operator $K:\mathcal{L}^0(\Omega)\times\mathcal{L}^0(\Omega)
\rightarrow\mathcal{L}^0(\Omega)$ defined by $K(X,Y):=X\cdot Y$ the product rule \eqref{prodrule} implies that 
\begin{align}\label{dkcon}
D^{(k)}_{z_1,\dotsc,z_k}(FG)&=\sum_{\abs{W}=k} K\bigl(D^{[W]}_{z_1,\dotsc,z_k}(F,G)\bigr)\,,
\end{align}
where the sum runs over all words $W$ of length $k$.
Hence, \eqref{hk} can be written as 
\begin{equation}\label{hk2}
h_k(z_1,\dotsc,z_k)=\frac{1}{k!}\sum_{\abs{W}=k}\E\Bigl[K\bigl(D^{[W]}_{z_1,\dotsc,z_k}(F,G)\bigr)\Bigr]\,.
\end{equation}
For a word $W=(W_1,\dotsc,W_m)\in\{L,R,B\}^m$ as above we define its \textit{characteristic} $\chi(W):=(l(W),r(W),b(W))$ by letting 
\begin{align*}
 l(W)&:=\babs{\{i\in[m]\,:\, W_i=L\}}\,,\\
 r(W)&:=\babs{\{i\in[m]\,:\, W_i=R\}}\quad\text{and}\\
 b(W)&:=\babs{\{i\in[m]\,:\, W_i=B\}}\,.
 \end{align*}
We call two words $W=(W_1,\dotsc,W_m), V=(V_1,\dotsc,V_m)\in\{L,R,B\}^m$ \textit{equivalent}, and write $W\sim V$, if $\chi(W)=\chi(V)$. For each $m\geq1$ this clearly defines an equivalence relation on the set of words of length $m$ and $W\sim V$ if and only if there is a permutation $\pi$ of $[m]$ such that $V=(W_{\pi(1)},\dotsc,W_{\pi(m)})$. In what follows we will write $\overline{W}$ for the equivalence class of the word $W$. 
Note that the random quantity $K\bigl(D^{[W]}_{z_1,\dotsc,z_k}(F,G)\bigr)$ is either equal to $0$, if $l(W)+b(W)>p$ or if $r(W)+b(W)>q$, or else is a product of quantities of the type $(p)_{p-s}I_s(u)$ and $(q)_{q-r}I_t(v)$, where $I_s(u)$ and 
$I_t(v)$ are two multiple integrals, $s=p-l(W)-b(W)$, $t=q-r(W)-b(W)$ and the kernels $u\in L^2(\mu^s)$ and $v\in L^2(\mu^t)$ depend on the variables $z_1,\dotsc,z_k$ (as usual, we use the notation $(n)_m=n(n-1)\cdot\ldots\cdot (n-m+1)=\frac{n!}{(n-m)!}$ to indicate the \textit{falling factorial}, defined for integers $0\leq m\leq n$).

\smallskip

 \noindent Note that, by orthogonality of multiple integrals of different orders, we have
\begin{equation*}
 \E[I_s(u)I_t(v)]=\delta_{s,t}s!\int_{\mathcal{Z}^s} \tilde{u}(x_1,\dotsc,x_s) \tilde{v}(x_1,\dotsc,x_s)d\mu^s(x_1,\dotsc,x_s)\,.
\end{equation*}
In particular, $\E[I_s(u)I_t(v)]\not=0$ only if $s=t$.

\smallskip

\noindent According to these facts, let us fix $0\leq k\leq p+q$ as well as a word 
$W=(W_1,\dotsc,W_k)\in\{L,R,B\}^k$ of characteristic $\chi(W)=(l,r,b)$ such that $p\geq l+b$, $q\geq r+b$ and $p-l=q-r$. Note that we have the identity $k=l+r+b$. We now aim at expressing 
\begin{equation*}
 h^{W}(z_1,\dotsc,z_k):=\frac{1}{k!}\sum_{V\in\Wbar}\E\Bigl[K\bigl(D^{[V]}_{z_1,\dotsc,z_k}(F,G)\bigr)\Bigr]
\end{equation*}
in a more explicit way. Firstly, it is clear from the definitions and from \eqref{Dm} that $h^W$ is a symmetric function of the variables $z_1,\dotsc,z_k\in \mathcal{Z}$. Indeed, we claim that 
\begin{align}\label{hw}
 h^W=\frac{p!q!}{l! r! b!(p-l-b)!}\widetilde{\bigl(f\star_{p-l}^{p-l-b}g\bigr)}
\end{align}
In order to see this, for $\sigma,\pi \in\mathbb{S}_k$, we define the following relation: Let 
\begin{align*}
 B_1&:=\{1,\dotsc,b\}\,,\quad B_2:=\{b+1,\dotsc,b+l\}\quad\text{and}\\
 B_3&:=\{b+l+1,\dotsc,b+l+r=k\}\,.
\end{align*}
Then, we write $\sigma\approx\pi$ if and only if $(\sigma\circ\pi^{-1})(B_j)=B_j$ for all $j=1,2,3$. Equivalently, $\sigma\approx\pi$ if, and only if, for all $i\in[k]$ and all $j=1,2,3$ it holds that 
$\sigma(i)\in B_j\Leftrightarrow\pi(i)\in B_j$. This clearly defines an equivalence relation on $\mathbb{S}_k$. It is easily checked that for $\sigma\in\mathbb{S}_k$ its equivalence class $[\sigma]$ has cardinality
\begin{equation*}
 \abs{[\sigma]}=\abs{B_1}!\cdot \abs{B_2}!\cdot\abs{B_3}!=b! l! r!
\end{equation*}
and, hence, there are exactly $m:=\frac{k!}{b! l! r!}$ equivalence classes. Let $\sigma_1,\dotsc,\sigma_m$ be a complete system of representatives for the relation $\approx$. It is easy to see that 
\begin{align}\label{coneq1}
 \widetilde{\bigl(f\star_{p-l}^{p-l-b}g\bigr)}(z_1,\dotsc,z_k)&=\frac{1}{k!}\sum_{\sigma\in\mathbb{S}_k}\bigl(f\star_{p-l}^{p-l-b}g\bigr)(z_{\sigma(1)},\dotsc,z_{\sigma(k)})\notag\\
&=\frac{b!r!l!}{k!}\sum_{i=1}^m \bigl(f\star_{p-l}^{p-l-b}g\bigr)(z_{\sigma_i(1)},\dotsc,z_{\sigma_i(k)})\,.
\end{align}
Now let $a:\{L,R,B\}\rightarrow\{0,b,b+l\}$ be defined by $a(B):=0$, $a(L):=b$ and $a(R):=b+l$. For each $V\in\Wbar$ we define a permutation $\sigma_V\in\mathbb{S}_k$ as follows: For $i\in\{1,\dotsc,k\}$ let 
\[n_i:=\babs{\{j\in\{1,\dotsc,i\}\,:\, V_j=V_i\}}\]
and define $\sigma_V(i):=a(V_i)+n_i$. It is easy to see that $\sigma_V$ is indeed a permutation on $\{1,\dotsc,k\}$ and that the mapping $V\mapsto[\sigma_V]$ is a bijection from $\Wbar$ to the set of equivalence classes with respect to $\approx$. Moreover, note that from the isometry formula for multiple integrals, for each $V\in\Wbar$ we have 
\begin{align*}
\E\Bigl[K\bigl(D^{[V]}_{z_1,\dotsc,z_k}(F,G)\bigr)\Bigr]&=(p)_{l+b}(q)_{r+b}(p-l-b)!\bigl(f\star_{p-l}^{p-l-b}g\bigr)(z_{\sigma_V(1)},\dotsc,z_{\sigma_V(k)})\\
&=\frac{(p-l-b)!p!q!}{(p-l-b)!(q-r-b)!}
\bigl(f\star_{p-l}^{p-l-b}g\bigr)(z_{\sigma_V(1)},\dotsc,z_{\sigma_V(k)})\\
&=\frac{p! q!}{(p-l-b)!}
\bigl(f\star_{p-l}^{p-l-b}g\bigr)(z_{\sigma_V(1)},\dotsc,z_{\sigma_V(k)})\,,
\end{align*}
where we have used that $p-l-b=q-r-b$.
Together with \eqref{coneq1} this gives 
\begin{align}\label{coneq2}
\widetilde{\bigl(f\star_{p-l}^{p-l-b}g\bigr)}(z_1,\dotsc,z_k)&=\frac{b!r!l!}{k!}\sum_{i=1}^m \bigl(f\star_{p-l}^{p-l-b}g\bigr)(z_{\sigma_i(1)},\dotsc,z_{\sigma_i(k)})\notag\\
&=\frac{b!r!l!(p-l-b)!}{p!q!k!}\sum_{V\in\Wbar}\E\Bigl[K\bigl(D^{[V]}_{z_1,\dotsc,z_k}(F,G)\bigr)\Bigr]\notag\\
&=\frac{b!r!l!(p-l-b)!}{p!q!}h^W(z_1,\dotsc,z_k)\,,
\end{align}
proving \eqref{hw}. Finally, observing that the characteristic $\chi(W)=(l(W),r(W),b(W))$ of a word $W=(W_1,\dotsc,W_k)$ of length $k$ is determined by $l(W)$ and $b(W)$ (since $r(W)=k-l(W)-b(W)$) we obtain that 
\begin{align*}
h_k&=\sum_{b=0}^{p\wedge q\wedge k}\sum_{l=0}^{p-b}\1_{\{k-l-b=q-(p-l)\}}\frac{p!q!}{l! (q-p+l)! b!(p-l-b)!}\widetilde{\bigl(f\star_{p-l}^{p-l-b}g\bigr)}\notag\\
&=\sum_{b=0}^{p\wedge q\wedge k}\sum_{s=b}^{p\wedge q}\1_{\{k+s-p-b=q-s\}}\frac{p!q!}{(p-s)! (q-s)! b!(s-b)!}\widetilde{\bigl(f\star_{s}^{s-b}g\bigr)}\,.
\end{align*}
Finally, for $m=0,1,\dotsc,2(p\wedge q)$ and with the change of variable $k=p+q-m$ we can rewrite this as 
\begin{align}\label{hpqm2}
h_{p+q-m}&=\sum_{s=\lceil\frac{m}{2}\rceil}^{m\wedge p\wedge q} \frac{p!q!}{(p-s)! (q-s)! (2s-m)!(m-s)!}\widetilde{\bigl(f\star_{s}^{m-s}g\bigr)}\,,
\end{align}
proving the forward implication of the theorem.

For the converse we make use of Lemma \ref{l2lemma} with $F$ replaced by $FG\in L^1(\P)$ and with $M=p+q$. Indeed, it is easy to see from \eqref{prodrule} that condition (a) of Lemma \ref{l2lemma} is satisfied with this choice of $M$. Moreover, the first part of condition (b) follows from the observation made in the beginning of this proof and the second part holds true by the assumptions on the kernels $h_{p+q-m}$, $m=0,\dotsc, 2(p\wedge q)$ and by a combination of the identities \eqref{hk} and \eqref{hpqm2}.

\end{proof}

\normalem
\bibliography{product}{}
\bibliographystyle{alpha}
\end{document}